\documentclass[12pt,reqno]{amsart}

\usepackage[mathlines]{lineno}

\usepackage{times}
\usepackage{amsfonts}
\usepackage{amssymb}
\usepackage{latexsym}
\usepackage{xcolor}

\newtheorem{theorem}{Theorem}
\newtheorem{lemma}[theorem]{Lemma}

\newtheorem{definition}[theorem]{Definition}

\newtheorem{corollary}[theorem]{Corollary}

\def\nin{\relax\hbox{$/\kern-.7em{\rm \in\,}$}}

\begin{document}

\begin{center}
\Large{ $\mu$-Hankel Operators on   Compact Abelian Groups}
\end{center}

\begin{center}
Adolf Mirotin
\end{center}

\textbf{Abstract.} $(\mu;\nu)$-Hankel operators between  separable Hilbert spaces were introduced and studied  recently (\textit{$\mu$-Hankel operators on Hilbert spaces}, Opuscula Math., \textbf{41}  (2021),  881--899). This  paper, is devoted to  generalization of $(\mu;\nu)$-Hankel operators to the  (non-separable in general) case of Hardy spaces over compact and connected Abelian groups. In this setting bounded $(\mu;\nu)$-Hankel operators are fully described  under some natural conditions. Examples of integral operators are considered.

\textbf{Key wards:} Hankel operator, $\mu$-Hankel operator, Hardy space, compact Abelian group, ordered group,  nuclear operator, integral operator.


\section{Introduction}

This paper is motivated by the following questions. Let $q\in \mathbb{C}$. The generalized Cauchy transform
\begin{equation*}
\mathbf{A}f(z)=\int_{\mathbb{D}}\frac{f(q\zeta)}{z-\zeta}d\sigma(\zeta) \quad (|z|=1)
\end{equation*}
 where $\sigma$ is a finite (in general complex) measure on the  unite disc $\mathbb{D}$ is an integral Hankel   operator for $q=1$. In this case  the boundedness of this operator in Hardy spaces on $\mathbb{D}$   were studied by Widom and others. But the case   $q\ne 1$, $0<|q|\le 1$ remained unexplored.
 
 The analogous  question can be posed for the another natural generalization of the Cauchy transform
\[
\mathbf{B}f(z):=\int_{\overline{\mathbb{D}}}\frac{f(\zeta)}{qz-\zeta }
d\tau(\zeta)\quad (|z|=1)
\]
where $|q|>1$ and similar operators.

Operators mentioned above are special cases of the so-called  $(\mu;\nu)$-Hankel operators (see Examples 1--3 below).
In \cite{Opuscula} $(\mu;\nu)$-Hankel operators $A=A_{(\mu;\nu),\alpha}$  between  separable Hilbert spaces $\mathcal{H}$ and $\mathcal{H}'$ where $\mu$ and $\nu$ are a nonzero complex parameters, $(\alpha_n)$ is a sequence of complex numbers were introduced as operators with matrices of the form $\langle Ae_k,e'_j\rangle=\mu^k\nu^j\alpha_{k+j}$ where $(e_k)$ and $(e'_j)$ are some orthonormal bases of  $\mathcal{H}$ and $\mathcal{H}'$ respectively. Conditions of boundedness of such operators were given. It was shown also that in the case $|\mu|\ne 1$ bounded $\mu$-Hankel operators are nuclear  and in the case $|\mu|= 1$ connections with Hankel operators were investigated.

This paper, is devoted to the non-separable case. In our opinion a convenient setting of the problem is as follows. Consider compact and connected Abelian group $G$ with a fixed  total group ordering in its dual $X$.
Let $H^2(G)$ be the corresponding Hardy space on $G$ \cite{Rud} and $H^2_-(G)=L^2(G)\ominus H^2(G)$. Then the positive cone $X_+$ and the negative cone $X_-=X\setminus X_+$ are orthogonal bases of $H^2(G)$ and $H^2_-(G)$ respectively. Since the Hilbert dimension of $H^2(G)$ and $H^2_-(G)$ equals to the  cardinality of $X$, this spaces are non-separable for non-metrizable $G$.  Consider a pare of non-null    homomorphisms $\mu, \nu\in \mathrm{Hom}(X,\mathbb{C})$ (here $\mathbb{C}$ is considered as a multiplicative  semigroup). We define a  $(\mu;\nu)$-Hankel operator between
 $H^2(G)$ and $H^2_-(G)$ as an operator whose matrix with respect to the bases $X_+$ and  $X_-$ is of the form $(\mu(\chi)\nu(\xi)a(\chi\xi))$ (here $\chi$ runs over $X_+$ and $\overline{\xi}$ runs over $X_-$). In this setting bounded $(\mu;1)$-Hankel operators (from now on called  $\mu$-Hankel operators) will be  fully described  under some natural conditions. The general case can be easily reduced to the case of  $\mu$-Hankel operators.  Examples of integral  $\mu$-Hankel operators are investigated.

 In the last section we briefly discuss the class of $(1;\nu)$-Hankel operators (called  $\nu$-Hankel operators)   that are in some sense "dual" to  $\mu$-Hankel operators.

  In the case $G=\mathbb{T}$ (the circle group), $X=\mathbb{Z}$ (the group of integers) we get the situation considered in \cite{Opuscula} but our main results are new  even in this case.

\section{Preliminaries}

This section collects all preliminary information we need in the next parts of the paper and the general definition of a $(\mu;\nu$-Hankel operator.

  In the following unless otherwise stated $G$ stands for a compact and connected Abelian group  and a total order (which agrees with the group structure) is fixed on its dual group $X$.  In turn, $X$ is the dual group for $G$ by the Pontryagin--van Kampen theorem.
  
  Let $\chi_0\equiv 1$ be the unit character and $X_+:=\{\chi\in X:\chi\geq \chi_0\}$   the positive cone in   $X$. In other words,  $X_+$ is a subsemigroup of $X$ such that $X_+^{-1}\cup X_+=X$, and $X_+^{-1}\cap X_+=\{\chi_0\}$  (see, e.g., \cite[Chapter 8]{Rud}). The inequality  $\chi\ge\xi$ for $\chi, \xi\in X$ is equivalent to the inclusion $\chi\xi^{-1}\in X_+$. We put also $X_-:=X\setminus X_+$. Then $X_-= X_+^{-1}\setminus \{\chi_0\}=\overline{X_+}\setminus \{\chi_0\}$ (the bar stands for the complex conjugation).

    As is well known, a (discrete) Abelian group $X$ can be totally ordered if and
only if it is torsion-free, which in turn is equivalent to the
condition that its character group $G$ is connected. In general the group $X$
may possess many different total orderings.

Note that groups $X$ for which $X_+$ has the smallest (first) element $\chi_1$ were described in [5,
Lemma 3.2]. In this  case we denote by $X^i$ the (infinite cyclic) subgroup of $X$ generated by $\chi_1$.

 In applications, often $X$ is a dense
 subgroup of $\mathbb{R}^d$ endowed with the discrete topology and $G=\mathbf{b}X$ is its Bohr compactification, or $X = \mathbb{Z}^d$ so that
$G = \mathbb{T}^d$ is the $d$-torus. Other interesting examples are the infinite dimensional torus $\mathbb{T}^\infty$, an  $\mathbf{a}$-adic solenoids $\Sigma_\mathbf{a}$, and their finite and countable products (see, e.g., \cite{DS2019}, \cite{MathNachr}).

\begin{definition}\label{def1} Let $G$ be a  compact Abelian group, $X$ its dual group. The Hardy space $H^2(G)$  on $G$ is the following subspace of $L^2(G)$ (see,
e.g., \cite{Rud}):
$$
H^2(G) = \{f\in  L^2(G) : \widehat{ f}(\chi) = 0 \forall \chi\in X_-\};
$$
where $\widehat{ f}$ denotes the Fourier transform of a function $f\in  L^1(G)$ the norm and the inner product in $L^2(G)$ (as
well as the norm and the inner product in $H^2(G)$) is denoted by $\langle \cdot,\cdot\rangle$ and $\|\cdot\|_2$ respectively.
\end{definition}

The orthogonal complement of $H^2(G)$ in $L^2(G)$ is denoted by $H^2_-(G)$.
 Then
$$
H^2_-(G) = \{f\in  L^2(G) : \widehat{ f}(\xi) = 0 \forall \xi\in X_+\}.
$$

The set $X_+$ forms an orthonormal basis in $H^2(G)$; while $X_-$ is an orthonormal basis in
$H^2_-(G)$
 (the Hilbert dimension of these spaces is equal to the cardinality of $X$). We denote
by $P_+$ and $P_-=I-P_+$ the orthogonal projections from $L^2(G)$ onto $H^2(G)$ and $H^2_-(G)$, respectively.

\begin{definition}\label{def2} Let $\varphi\in L^2(G)$. An operator $H_\varphi : H^2(G)\to H^2_-(G)$
defined by
$$
H_\varphi f=P_+(\varphi f)
$$
is called a \textit{Hankel operator on $G$ with symbol} $\varphi$.
\end{definition}

For $\varphi\in L^\infty(G)$ this operator is bounded.
The theory of such operators was developed in \cite{YCG}, \cite{aA}, \cite{Trudy}, \cite{Trudy2} (see also \cite{Indag}).

In this paper, we consider the next generalization of Hankel operators over $G$.

\begin{definition}\label{def3} Consider a non-null homomorphisms $\mu$ and $\nu$ of the semigroup $X_+$ to a multiplicative semigroup $\mathbb{C}$ (a generalized semi-characters of $X_+$) and a function $a: X_+\setminus\{\chi_0\}\to\mathbb{C}$. A linear operator $A=A_{(\mu;\nu),a}$ between the spaces  $H^2(G)$ and  $H^2_-(G)$
that is  defined at least on the linear subspace $\mathrm{span}(X_+)$ of $H^2(G)$ and satisfies
$$
\langle A_{(\mu;\nu),a}\chi,\overline{\xi}\rangle=\mu(\chi)\nu(\xi)a(\chi\xi)\ \ \forall \chi\in X_+,  \xi\in  X_+\setminus\{\chi_0\}
$$
will be called a \textit{ $(\mu;\nu)$-Hankel operator}.
\end{definition}

We put  $A_{\mu,a}:=A_{(\mu;1),a}$ and call such operator a \textit{$\mu$-Hankel operator}. If $\nu(\xi)\ne 0$ $\forall  \xi\in  X_+$ we have $A_{(\mu;\nu),a}=A_{\left(\frac{\mu}{\nu};1\right),a\nu}$. That is  why the main part of the paper is devoted to $\mu$-Hankel operators. We put  also $B_{\nu,a}:=A_{(1;\nu),a}$ and call such operator a \textit{$\nu$-Hankel operator}.  In the last section of the paper $\nu$-Hankel operators are also considered in brief.

We use also following notation. Let $H_1$ and $H_2$ be two Hilbert spaces. Given two vectors $y\in H_1$, and $f\in H_2$ we define the rank-one operator $f\otimes y:H_1\to H_2$
by
$$
(f\otimes y)x=
\langle x,y\rangle f, \mbox{ for all } x\in H_1.
$$
We will use throughout this paper the following easy-to-check properties of
this operator
$(i) A(f\otimes y)B=(Af)\otimes (B^*y) \mbox{ for  } A\in \mathcal{L}(H_2), B\in \mathcal{L}(H_1)$,
$(ii) (f\otimes y)^*=y\otimes f$, and $(iii) \|f\otimes y\|=\|f\|\|y\|$.

As usual we denote by $\mathcal{L}(V,W)$ the space of all bounded linear operators between Banach spaces $V$ and $W$, $\mathcal{L}(V):=\mathcal{L}(V,V)$.

\section{$\mu$-Hankel Operators on  G}

\subsection{Conditions of boundedness  of $\mu$-Hankel operators and the generalized Hankel equation}

\

\textbf{Remark 1.} Note that $\overline{\xi}$ runs over the orthonormal basis $X_-$ of $ H^2_-(G)$ in the  definition \ref{def3}.
Thus, for each $\chi\in X_+$ the numbers $\mu(\chi)a(\chi\xi)$
are Fourier coefficients of the function $A_{\mu,a}\chi$ with respect to this basis. Taking $\chi=\chi_0$  we conclude that $\|A_{\mu,a}\chi_0\|=\|a\|_{\ell^2(X_+\setminus\{\chi_0\})}$ by Parseval's formula and thus $a(\xi)\ne 0$ for an at most countable set   of $\xi\in X_+$.

\begin{theorem}\label{thm01}
 Let $\mu\in\ell^2(X_+)$. A  $\mu$-Hankel operator $A_{\mu,a}$ extends from $\mathrm{span}(X_+)$ to a bounded  operator between the spaces  $H^2(G)$ and  $H^2_-(G)$
if and only if $a\in \ell^2(X_+\setminus\{\chi_0\})$. In this case the extension is unique and
$$
\|A_{\mu,a}\|\le \|\mu\|_{\ell^2(X_+)}\|a\|_{\ell^2(X_+\setminus\{\chi_0\})}.
$$
\end{theorem}

\begin{proof}
 The necessity of the condition $a\in \ell^2(X_+\setminus\{\chi_0\})$  was mentioned in Remark 1.

Now let $\mu, a\in \ell^2(X_+\setminus\{\chi_0\})$ and $A=A_{\mu,a}$. For each polynomial $f\in \mathrm{span}(X_+)$ we have $f=\sum_{\chi\in X_+}\langle f,\chi\rangle \chi$ (a finite sum) and thus $Af=\sum_{\chi\in X_+}\langle f,\chi\rangle A\chi$. It follows that
\begin{eqnarray*}
Af&=&\sum_{\xi\in X_+\setminus\{\chi_0\}}\langle Af,\overline{\xi}\rangle \overline{\xi}=\sum_{\xi\in X_+\setminus\{\chi_0\}}\left(\sum_{\chi\in X_+}\langle f,\chi\rangle \langle A\chi,\overline{\xi}\rangle\right)\overline{\xi}\\
&=&\sum_{\xi\in X_+\setminus\{\chi_0\}}\left(\sum_{\chi\in X_+}\langle f,\chi\rangle \mu(\chi)a(\chi\xi)\right)\overline{\xi}.
\end{eqnarray*}
Therefore by Parseval's formula and Cauchy-Bunyakovskii-Schwartz inequality one has
\begin{eqnarray*}
\|Af\|^2&=&\sum_{\xi\in X_+\setminus\{\chi_0\}}\left|\sum_{\chi\in X_+}\langle f,\chi\rangle \mu(\chi)a(\chi\xi)\right|^2\\
&\le&\sum_{\xi\in X_+\setminus\{\chi_0\}}\left(\sum_{\chi\in X_+}|\mu(\chi)|^2\sum_{\chi\in X_+}|a(\chi\xi)\langle f,\chi\rangle|^2\right)\\
&=&\|\mu\|_{\ell^2}^2\sum_{\chi\in X_+}|\langle f,\chi\rangle|^2\sum_{\xi\in X_+\setminus\{\chi_0\}}|a(\chi\xi)|^2\\
&\le&\|\mu\|_{\ell^2(X_+)}^2\|a\|_{\ell^2(X_+\setminus\{\chi_0\})}^2\|f\|^2.
\end{eqnarray*}
To finish the proof
 it remains to note that $\mathrm{span}(X_+)$ is dense in $H^2(G)$ by \cite[Lemma 1]{MathSb}\footnote{There is a typo in the proof of this lemma. It should to be $H^p_T(G)$ instead of $H^p(G)$ in the beginning of the proof.}.

\end{proof}

\begin{theorem}\label{thm1} A bounded  operator $A: H^2(G)\to H^2_-(G)$  is  $\mu$-Hankel if and only if it satisfies the \textit{generalized Hankel equation}
    \begin{equation}\label{GHE}
    AS_\chi=\mu(\chi) P_{-} \mathcal{S}_\chi A \ \ \forall \chi\in X_+, 
    \end{equation}
where $\mathcal{S}_\chi f:=\chi f$ for all  $f\in L^2(G)$ is a bilateral shift operator  on $L^2(G)$ with respect to $\chi$, $S_\chi:=\mathcal{S}_\chi|H^2(G)$ a unilateral shift in $H^2(G)$.
\end{theorem}

\begin{proof} Let an operator $A_{\mu,a} : H^2(G)\to H^2_-(G)$ be  bounded. Then for  $\chi, \eta\in X_+,  \xi\in  X_+\setminus\{\chi_0\}$ we have
\begin{eqnarray*}
\langle A_{\mu,a}S_\chi \eta,\overline{\xi}\rangle&=&\langle A_{\mu,a}(\chi \eta),\overline{\xi}\rangle=\mu(\chi)(\mu(\eta)a(\eta(\chi\xi)))\\
&=&\mu(\chi)\langle A_{\mu,a}\eta,\overline{\chi \xi}\rangle
=\mu(\chi)\langle \chi A_{\mu,a}\eta,\overline{\xi}\rangle\\
&=&\mu(\chi)\langle \chi A_{\mu,a}\eta,P_-\overline{\xi}\rangle=\langle \mu(\chi) P_{-} \mathcal{S}_\chi A_{\mu,a} \eta,\overline{\xi}\rangle.
\end{eqnarray*}
Since $A_{\mu,a}$ is bounded, (\ref{GHE}) follows.

Conversely, let  (\ref{GHE}) holds for a bounded  operator $A: H^2(G)\to H^2_-(G)$. We put $a(\xi):=\langle A\chi_0,\overline{\xi}\rangle$
for $\xi\in  X_+$. Since $\chi=S_\chi \chi_0$, one has for $\chi\in X_+,  \xi\in  X_+\setminus\{\chi_0\}$ that
\begin{eqnarray*}
\langle A\chi,\overline{\xi}\rangle&=&\langle AS_\chi \chi_0,\overline{\xi}\rangle=\langle \mu(\chi) P_-\mathcal{S}_\chi A\chi_0,\overline{\xi}\rangle\\
&=&\mu(\chi)\langle  \mathcal{S}_\chi A\chi_0, P_-\overline{\xi}\rangle
=\mu(\chi)\langle   A\chi_0, \overline{\chi\xi}\rangle= \mu(\chi)a(\chi\xi).
\end{eqnarray*}
\end{proof}

\subsection{The unimodular case}

\begin{theorem}\label{thm2} Let $|\mu(\chi)|=1 \ \ \forall \chi\in X_+$. An  operator   $A$ between the spaces  $H^2(G)$ and  $H^2_-(G)$ is  bounded and $\mu$-Hankel if and only if $A=H_\varphi U$ for some unitary operator $Uf(x):=f(gx)$ in $H^2(G)$ ($g\in G$) and a Hankel operator $H_\varphi$ on $G$  with $\varphi\in L^\infty(G)$. In this case,
$$
\|A\|=\mathrm{dist}_{L^\infty}(\varphi, H^\infty(G)).
$$
\end{theorem}

\begin{proof} Assume that $A=A_{\mu,a}$ and this operator is  bounded. The rule $\widetilde{\mu}(\chi^{-1}\chi'):=\mu(\chi)^{-1}\mu(\chi')$
where $\chi, \chi'\in X_+$ extends $\mu$ from $X_+$ to a character $\widetilde{\mu}$ of the group $X$. By the Pontryagin--van Kampen theorem there is such an element $g\in G$ that $\widetilde{\mu}(\chi)=\chi(g)$ for all $\chi\in X$. The translation operator $Uf(x):=f(gx)$ acts in $H^2(G)$  and is unitary. Consider a bounded operator $H:H^2(G)\to H^2_-(G)$, $H:=A_{\mu,a}U^{-1}$. Then for all $\chi\in X_+,  \xi\in  X_+\setminus\{\chi_0\}$ one has
\begin{eqnarray*}
\langle H\chi,\overline{\xi}\rangle&=&\langle A_{\mu,a}U^{-1}\chi,\overline{\xi}\rangle
=\langle \chi(g^{-1})A_{\mu,a}\chi,\overline{\xi}\rangle\\
&=&\chi(g)^{-1}\mu(\chi)a(\chi\xi)=a(\chi\xi).
\end{eqnarray*}
It follows that for all $\chi, \chi'\in X_+,  \xi\in  X_+\setminus\{\chi_0\}$
$$
\langle HS_\chi\chi',\overline{\xi}\rangle=a(\chi\chi'\xi)=\langle H\chi',\overline{\chi\xi}\rangle=\langle \chi H\chi',P_-\overline{\xi}\rangle=\langle P_-\mathcal{S}_\chi  H \chi',\overline{\xi}\rangle.
$$
Thus, the Hankel equation $HS_\chi=P_-\mathcal{S}_\chi H$ holds for all $\chi\in X_+$
 and therefore by a version of the Nehari--Wang theorem \cite[Theorem 1.5]{aA} the operator $H$ is Hankel on $G$ (and vice versa),  $H=H_\varphi$ for some $\varphi\in L^\infty(G)$, and
$$
\|A_{\mu,a}\|=\|HU\|=\|H_\varphi\|=\mathrm{dist}_{L^\infty}(\varphi, H^\infty(G)).
$$
This proves the necessity.

To prove the sufficiency, note that by the Hankel equation for $H_\varphi$ \cite[Theorem 1.5]{aA} we have  $H_\varphi S_\chi U=P_-\mathcal{S}_\chi H_\varphi U$  for all $\chi\in X_+$.
  Since $S_\chi U=\chi(g)^{-1}US_\chi=\mu(\chi)^{-1}US_\chi$, it follows that the operator $H_\varphi U$ satisfies (\ref{GHE}) and Theorem \ref{thm1} works.
\end{proof}

 The previous theorem enables us to use for $\mu$-Hankel operators  in the unimodular case  results on Hankel operators over $G$ from \cite{aA}. In particular, we get the following corollary.

\begin{corollary}\label{cor1}  Let $|\mu(\chi)|=1 \ \ \forall \chi\in X_+$. A non-trivial compact $\mu$-Hankel operator on $G$ exists if and only if there is a first positive element $\chi_1\in X_+$.
\end{corollary}

 \begin{proof} This follows from the previous theorem and the fact that a non-trivial compact Hankel operator on $G$ exists if and only if there is a first positive element in $X$ \cite{YCG} (see also \cite[Theorem 2.1]{aA}).
\end{proof}

\subsection{The non-unimodular case}

In this subsection we assume that  a group $X$ possesses the first positive element $\chi_1$
and  $\chi_1^0:=\chi_0$. By $H^2(\chi_1)$ we denote the Hilbert subspace of $H^2(G)$ with orthonormal
 basis $\{\chi_1^n:n\in \mathbb{Z}_+\}$.
We put $ X_+^i:=X^i\cap X_+=\{\chi_1^n: n\in \mathbb{Z}_+\}$.
\begin{lemma}\label{lem1} (\cite{aA}) The set $X_+ \setminus X_+^i$ is an ideal of the semigroup $X_+$.
\end{lemma}

\begin{proof} First note that  $\chi\in X_+ \setminus X_+^i$ if and only if $\chi>\chi_1^n\  \forall n\in \mathbb{Z}_+$. Indeed, the "if" part is obvious. Now  assume that  $\chi\in X_+ \setminus X_+^i=X_+ \setminus X^i$ but $\chi\le \chi_1^n$ for some $n\in \mathbb{Z}_+$. Since $\chi_0\le \chi$ and the group $X^i$ is convex \cite[Theorem 2]{MathSb}, we get $\chi\in X^i$, a contradiction.

Thus, for all $\xi\in X_+$ and  $\chi\in X_+ \setminus X_+^i$ we have $\xi\chi>\xi\chi_1^n \forall n\in \mathbb{Z}_+$ and therefore $\xi\chi\in X_+ \setminus X_+^i$. The proof is complete.
\end{proof}

 The next theorem describes bounded $\mu$-Hankel operators in the case where $|\mu(\chi_1)|<1$.
  In the following lemma we use the techniques found in \cite{Sun} (see also \cite{Avendano}). Recall that $S_{\chi}=\mathcal{S}_{\chi}|H^2(G)$ for $\chi\in X_+$.

 \begin{lemma}\label{lemeq1} Let $Q\in \mathcal{L}(H^2_-(G))$, $\|Q\|<1$. Then an operator $X$ from  $\mathcal{L}(H^2(G),H^2_-(G))$ is a solution of the operator equation
\begin{equation}\label{eq1}
QX=XS_{\chi_1}
\end{equation}
if and only if  it has the  form
\begin{equation}\label{eq10}
X=\sum_{n=0}^\infty (Q^n \varphi_0)\otimes \chi_1^{n}
\end{equation}
for some function $\varphi_0\in H^2_-(G)$.
\end{lemma}

\begin{proof} If $X$ satisfies (\ref{eq1}) then  $QXS_{\chi_1}^*=XS_{\chi_1}S_{\chi_1}^*$.
Since $S_{\chi}^*f=P_+(\overline{\chi}f)$ for $\chi\in X_+$, $f\in H^2(G)$, we have $S_{\chi_1}S_{\chi_1}^*=I_{H^2}-\chi_0\otimes \chi_0$ because both sides here coincide on $X_+$.
Therefore $QXS_{\chi_1}^*=X(I_{H^2}-\chi_0\otimes \chi_0$). It follows that
\[
(\mathcal{I}-\mathcal{T})X=\varphi_0\otimes \chi_0,
\]
where $\mathcal{T}X:=QXS_{\chi_1}^*$ is a linear transformation from $\mathcal{L}(\mathcal{L}(H^2(G),H^2_-(G)))$, $\|\mathcal{T}\|\le \|Q\|<1$, $\mathcal{I}$ is the unit operator  that acts in $\mathcal{L}(\mathcal{L}(H^2(G),H^2_-(G)))$, and $\varphi_0:=X\chi_0$.
Thus,
\[
X=(\mathcal{I}-\mathcal{T})^{-1}(\varphi_0\otimes \chi_0)=\sum_{n=0}^\infty \mathcal{T}^{n}(\varphi_0\otimes \chi_0)=\sum_{n=0}^\infty (Q^n \varphi_0)\otimes \chi_1^{n}.
\]

Conversely, let (\ref{eq10}) holds. Since $S_{\chi_1}^*=P_+\mathcal{S}_{\overline{\chi_1}}|H^2(G)$, and $(\varphi_0\otimes \chi_0)S_{\chi_1}=\varphi_0\otimes S_{\chi_1}^*\chi_0=\varphi_0\otimes (P_+\overline{\chi_1})=0$, we get
\begin{eqnarray*}
QX&=&\sum_{n=0}^\infty (Q^{n+1} \varphi_0)\otimes \chi_1^{n}\\
&=&\sum_{n=0}^\infty (Q^{n+1} \varphi_0)\otimes S_{\chi_1}^*(\chi_1^{n+1})=\sum_{k=1}^\infty ((Q^{k} \varphi_0)\otimes \chi_1^{k})S_{\chi_1}\\
&=&\sum_{k=1}^\infty ((Q^{k} \varphi_0)\otimes \chi_1^{k})S_{\chi_1}+(\varphi_0\otimes \chi_0)S_{\chi_1}\\
&=&\left(\sum_{n=0}^\infty (Q^{n} \varphi_0)\otimes \chi_1^{n}\right)S_{\chi_1}=XS_{\chi_1}.
\end{eqnarray*}

\end{proof}

\begin{theorem}\label{thm3}  Let a $\mu$-Hankel operator   $A=A_{\mu,a}$ is  bounded and $q:=\mu(\chi_1)$, $0<|q|<1$. Then
\begin{equation}\label{9}
A=\sum_{n=0}^\infty q^n \varphi_n\otimes \chi_1^n
\end{equation}
where   $\varphi_n$ is a sequence of functions from $H^2_-(G)$ with Fourier coefficients
\begin{equation}\label{15}
\langle \varphi_n,\overline{\xi} \rangle=a(\chi_1^n\xi) \ \ \forall n\in \mathbb{Z}_+, \xi\in  X_+\setminus\{\chi_0\}
\end{equation}
that is uniformly bounded in norm.
Moreover, $A|H^2(\chi_1)^\bot=0$,  operator $A$ is nuclear and its norm satisfies  $\|A\|\le \|\varphi_0\|/(1-|q|)$.

Conversely, if a sequence $\varphi_n\in H^2_-(G)$ is uniformly bounded in norm, and (\ref{15}) holds each  operator of the form (\ref{9}) with some $q\in \mathbb{D}$ is nuclear and equals to $A_{\mu,a}$, where  $\mu(\chi_1)=q$ and $\mu(\chi)=0$ for all $\chi\in X_+ \setminus X^i$.
\end{theorem}

\begin{proof} Necessity. The equation (\ref{GHE}) implies that
\begin{equation}\label{20}
\mu(\chi_1) P_{-} \mathcal{S}_{\chi_1}A=AS_{\chi_1}.
\end{equation}
Putting $X=A$ and $Q=\mu(\chi_1) P_{-} \mathcal{S}_{\chi_1}$ in Lemma \ref{lemeq1} we get from (\ref{20}) and (\ref{eq10}) that

\begin{eqnarray}\label{50}
A=\sum_{n=0}^\infty \mu(\chi_1)^n (P_{-} \mathcal{S}_{\chi_1})^n\varphi_0\otimes S_{\chi_1}^n\chi_0=\sum_{n=0}^\infty \mu(\chi_1)^n \varphi_n\otimes \chi_1^n,
\end{eqnarray}
where $\varphi_0=A\chi_0$, $\varphi_n:=(P_{-} \mathcal{S}_{\chi_1})^n\varphi_0$, $\|\varphi_n\|\le  \|(P_{-} \mathcal{S}_{\chi_1})\|^n\|\varphi_0\|\le \|\varphi_0\|$.

One  can prove by induction that  $\langle \varphi_n,\overline{\xi} \rangle=a(\chi_1^n\xi)$ ($n\in \mathbb{Z}_+$). Indeed, for $n=0$ this is true due to the equalities $\varphi_0=A\chi_0$ (see the proof of Lemma \ref{lemeq1}) and $\chi_1^0=\chi_0$. And the induction step follows from the equality
\begin{eqnarray*}
\langle \varphi_{n+1},\overline{\xi} \rangle&=&\langle (P_{-} \mathcal{S}_{\chi_1})^{n+1}\varphi_0,\overline{\xi} \rangle=\langle \chi_1(P_{-} \mathcal{S}_{\chi_1})^n\varphi_0, P_-\overline{\xi} \rangle \nonumber\\
&=&\langle (P_{-} \mathcal{S}_{\chi_1})^n\varphi_0,  \overline{\chi_1\xi} \rangle=\langle \varphi_{n},\overline{\chi_1\xi} \rangle.
\end{eqnarray*}

Formula (\ref{9}) shows that $Af=0$ for  $f\bot H^2(\chi_1)$ and thus $A|H^2(\chi_1)^\bot=0$. Finally, since the series in (\ref{9}) in one-dimensional operators converges absolutely in the operator norm, $A$ is nuclear and
$$
\|A\|\le \sum_{n=0}^\infty |q|^n \|\varphi_n\otimes \chi_1^n\|\le \sum_{n=0}^\infty |q|^n\|(P_{-} \mathcal{S}_{\chi_1})^n\varphi_0\|\|\chi_1^n\|\le \frac{\|\varphi_0\|}{1-|q|}.
$$

Sufficiency. Let (\ref{9}) holds with some $q\in \mathbb{D}$ and $\varphi_n\in H^2_-(G)$,  $\varphi_n$ are uniformly bounded in norm, and
  (\ref{15}) is valid for some $a:X_+\to \mathbb{C}$. Then $A$ is a nuclear operator between $H^2(G)$ and $H^2_-(G)$, $A|H^2(\chi_1)^\bot=0$ and   $A\chi_1^k=q^k\varphi_k$ for all   $k\in \mathbb{Z}_+$. Define $\mu:X_+\to \mathbb{D}$ by the rule $\mu(\chi_1^n)=q^n$ for $n\in \mathbb{Z}_+$ and $\mu|(X_+\setminus X^i)=0$. Then $\mu$ is  a semi-character by Lemma \ref{lem1} and we have
$$
\langle A\chi_1^k,\overline{\xi}\rangle=q^k\langle\varphi_k,\overline{\xi}\rangle=\mu(\chi_1^k)a(\chi_1^k\xi)
$$
for $k\in \mathbb{Z}_+$ and $\xi\in X_+\setminus\{\chi_0\}$.
Finally, $\langle A\chi,\overline{\xi}\rangle=0=\mu(\chi)a(\chi\xi)$ if $\chi\in X_+\setminus X^i$. Thus, $A=A_{\mu,a}$. This completes the proof.
\end{proof}

In the following for $\chi\in X_+$ we put $S_{\overline{\chi}}:=\mathcal{S}_{\overline{\chi}}|H^2_-(G)$.

\begin{lemma}\label{lemeq2} Let $R\in \mathcal{L}(H^2(G))$, $\|R\|<1$. Then an operator $Y$ from  $\mathcal{L}(H^2_-(G),H^2(G))$ is a solution of the operator equation
\begin{equation}\label{eq2}
RY=YS_{\overline{\chi_1}}
\end{equation}
if and only if  it has the  form
\begin{equation}\label{eq42}
Y=\sum_{n=0}^\infty (R^n \psi_1)\otimes \overline{\chi_1}^{n+1}
\end{equation}
for some function $\psi_1\in H^2(G)$.
\end{lemma}

\begin{proof} The proof is similar to the proof of Lemma \ref{lemeq1}. If $Y$ satisfies (\ref{eq2}) then  $RYS_{\overline{\chi_1}}^*=YS_{\overline{\chi_1}}S_{\overline{\chi_1}}^*$. But
$S_{\overline{\chi_1}}S_{\overline{\chi_1}}^*=I_{H^2_-}-\overline{\chi_1}\otimes \overline{\chi_1}$ (since $S_{\overline{\chi_1}}^*=P_-\mathcal{S}_{\chi_1}|H^2_-(G)$, both sides of this equation coincide on $X_-$). So, we have $RYS_{\overline{\chi_1}}^*=Y(I_{H^2_-}-\overline{\chi_1}\otimes \overline{\chi_1})$ or equivalently $Y-RYS_{\overline{\chi_1}}^*=\psi_1\otimes \overline{\chi_1}$
where $\psi_1:=Y\overline{\chi_1}$. This means that
\begin{equation}\label{eq44}
(\mathcal{I}_1-\mathcal{T}_1)Y=\psi_1\otimes \overline{\chi_1}
\end{equation}
where $\mathcal{T}_1Y:=RYS_{\overline{\chi_1}}^*$ is an operator valued transformation from the space $\mathcal{L}(\mathcal{L}(H^2_-(G),H^2(G)))$ and $\mathcal{I}_1$ stands for the identity in $\mathcal{L}(\mathcal{L}(H^2_-(G),H^2(G)))$. Since $\|\mathcal{T}_1\|\le \|R\|<1$, (\ref{eq44}) implies that
\begin{equation}\label{eq46}
Y=(\mathcal{I}_1-\mathcal{T}_1)^{-1}(\psi_1\otimes \overline{\chi_1})=\sum_{n=0}^\infty \mathcal{T}_1^{n}(\psi_1\otimes \overline{\chi_1})=\sum_{n=0}^\infty (R^n \psi_1)\otimes \overline{\chi_1}^{n+1}.
\end{equation}

Conversely, let (\ref{eq42}) holds. Since $S_{\overline{\chi_1}}^*=P_-\mathcal{S}_{\chi_1}|H^2_-(G)$, and $(\psi_1\otimes \overline{\chi_1})S_{\chi_1}=\psi_1\otimes S_{\chi_1}^*\overline{\chi_1}=\psi_1\otimes (P_-\overline{\chi_0})=0$, we get
\begin{eqnarray*}
RY&=&\sum_{n=0}^\infty (R^{n+1} \psi_1)\otimes \overline{\chi_1}^{n+1}\\
&=&\sum_{k=1}^\infty (R^{k} \psi_1)\otimes \overline{\chi_1}^{k}=\sum_{k=1}^\infty (R^{k} \psi_1)\otimes S_{\overline{\chi_1}}^*\overline{\chi_1}^{k+1}\\
&=&\sum_{k=1}^\infty (R^{k} \psi_1)\otimes \overline{\chi_1}^{k+1}S_{\overline{\chi_1}}+(\psi_1\otimes \overline{\chi_1})S_{\chi_1}\\
&=&\left(\sum_{n=0}^\infty (R^n \psi_1)\otimes \overline{\chi_1}^{n+1}\right)S_{\chi_1}=YS_{\chi_1}.
\end{eqnarray*}
\end{proof}

Now we shall describe bounded $\mu$-Hankel operators in the case where $|\mu(\chi_1)|>1$. In the following theorem $H^2(\overline{\chi_1})$  denotes the Hilbert subspace of $H^2_-(G)$ with orthonormal
 basis $\{\overline{\chi_1}^n:n\in \mathbb{N}\}$, and $\mathrm{Im} A$ stands for the image of an operator $A$.

\begin{theorem}\label{thm4}  Let a $\mu$-Hankel operator   $A=A_{\mu,a}$ is  bounded and $|\mu(\chi_1)|>1$. Then
\begin{equation}\label{11}
A=\sum_{n=1}^\infty\left(\frac{1}{\mu(\chi_1)}\right)^{n-1} \overline{\chi_1}^n \otimes \psi_n
\end{equation}
where   $\psi_n$ is a sequence of functions from $H^2(G)$ with Fourier coefficients
\begin{equation}\label{21}
\langle \psi_n,\chi \rangle=\overline{\mu(\chi_1^{n-1}\chi)a(\chi_1^{n}\chi)} \forall n\in \mathbb{N}, \chi\in  X_+
\end{equation}
that is uniformly bounded in norm.
Moreover, $\mathrm{Im} A\subset H^2(\overline{\chi_1})$,  the function $a$ is supported in $X_+^i\setminus\{\chi_0\}$,  operator $A$ is nuclear  and its  norm satisfies  $\|A\|\le |\mu(\chi_1)|\|\psi_1\|/(|\mu(\chi_1)|-1)$.

Conversely, if a sequence $\psi_n\in H^2(G)$ is uniformly bounded in norm, and (\ref{21}) holds with some  $\mu\in \mathrm{Hom}(X_+,\mathbb{C})$, $|\mu(\chi_1)|>1$ and some function $a:X_+\to \mathbb{C}$ supported in $X_+^i\setminus\{\chi_0\}$, then each  operator of the form (\ref{11})  is nuclear and equals to $A_{\mu,a}$.
\end{theorem}

\begin{proof} Necessity. The equation (\ref{GHE}) implies that $AS_{\overline{\chi_1}}=\mu(\chi_1) P_{-} \mathcal{S}_{\chi_1} A$. Or after passing to  conjugates
\begin{equation}\label{52}
S_{\overline{\chi_1}}^*A^*=\overline{\mu(\chi_1)}(P_{-} \mathcal{S}_{\chi_1} A)^*.
\end{equation}
 We claim that
$(P_{-} \mathcal{S}_{\chi_1} A)^*=A^*S_{\overline{\chi_1}}$.
Indeed, since $S_{\overline{\chi_1}}=(P_-\mathcal{S}_{\chi_1}|H^2_-)^*$, for each $\chi\in X_+$, $\xi\in X_+\setminus\{\chi_0\}$ one has
\[
\langle\chi, A^*S_{\overline{\chi_1}}\overline{\xi} \rangle=\langle\chi, A^*(P_-\mathcal{S}_{\chi_1}|H^2_-)^*\overline{\xi} \rangle=
\langle P_-\mathcal{S}_{\chi_1}A\chi,\overline{\xi} \rangle.
\]
Now (\ref{52}) takes the form $(\overline{\mu(\chi_1)})^{-1}S_{\overline{\chi_1}}^*A^*=A^*S_{\overline{\chi_1}}$.
Putting  $Y=A^*$, $R=(\overline{\mu(\chi_1)})^{-1}S_{\overline{\chi_1}}^*$ in Lemma \ref{lemeq2}
we conclude that
\[
A^*=\sum_{n=0}^\infty \frac{1}{\overline{\mu(\chi_1)}^{n}} ((S_{\overline{\chi_1}}^n)^{\ast}\psi_1)\otimes \overline{\chi_1}^{n+1}=\sum_{n=0}^\infty \left(\frac{1}{\overline{\mu(\chi_1)}}\right)^n(\psi_{n+1}\otimes \overline{\chi_1}^{n+1}),
\]
where $\psi_1=A^*\overline{\chi_1}$ (see the proof of Lemma \ref{lemeq2}), $\psi_{n+1}:=(S_{\overline{\chi_1}}^n)^*\psi_1$ and formula (\ref{11}) follows. The inclusion  $\mathrm{Im} A\subset H^2(\overline{\chi_1})$ is a direct consequence of (\ref{11}).

Further, for  all $\xi\notin X_+^i\setminus\{\chi_0\}$ we have
\[
\langle  (\overline{\chi_1}^n \otimes \psi_n)\chi_0, \overline{\xi}\rangle=\langle \chi_0, \psi_n\rangle\langle \overline{\chi_1}^n, \overline{\xi}\rangle=0.
\]
It follows  in view of (\ref{11}) that
\[
a(\xi)=\mu(\chi_0)a(\chi_0\xi)=\langle A\chi_0,\overline{\xi}\rangle=0.
\]
Thus the function $a$ is supported in $X_+^i\setminus\{\chi_0\}$.

To prove (\ref{21}) consider $ n\in \mathbb{Z}_+, \chi\in  X_+$. Then
\begin{eqnarray*}
\langle \psi_{n+1},\chi \rangle&=&\langle (S_{\chi_1}^n)^*\psi_1,\chi \rangle=\langle \psi_1,S_{\chi_1}^n\chi \rangle =\langle A^*\overline{\chi_1},\chi_1^n\chi\rangle\\
&=&\langle \overline{\chi_1}, A(\chi_1^n\chi)\rangle=\overline{\langle A(\chi_1^n\chi), \overline{\chi_1} \rangle}=\overline{\mu(\chi_1^{n}\chi)a(\chi_1^{n+1}\chi)}.
\end{eqnarray*}

For all $ n\in \mathbb{N}$ we have also $\|\psi_{n}\|\le \|(S_{\chi_1}^{n-1})^*\|\|\psi_1\|\le \|\psi_1\|$.

Moreover, $A$ is nuclear and
\begin{eqnarray*}
\|A\|&\le&  \sum_{n=1}^\infty\left(\frac{1}{|\mu(\chi_1)|}\right)^{n-1} \|\overline{\chi_1}^n\|\|\psi_n\|\\
&\le& \sum_{n=1}^\infty\left(\frac{1}{|\mu(\chi_1)|}\right)^{n-1} \|\psi_1\|=\frac{|\mu(\chi_1)|}{|\mu(\chi_1)|-1}\|\psi_1\|.
\end{eqnarray*}

Sufficiency. Let (\ref{11}) and (\ref{21}) hold  and the function $a$ is supported in $X_+^i\setminus\{\chi_0\}$. For  all $\chi\in X_+$, $\xi\in X_+\setminus\{\chi_0\}$ we have
\[
\langle A\chi,\overline{\xi}\rangle=\sum_{n=1}^\infty\left(\frac{1}{\mu(\chi_1)}\right)^{n-1} \langle(\overline{\chi_1}^n \otimes \psi_n)\chi,\overline{\xi}\rangle.
\]
But
\[
(\overline{\chi_1}^n \otimes \psi_n)\chi=\overline{\langle \psi_n,\chi\rangle}\overline{\chi_1}^n=\mu(\chi_1^{n-1}\chi)a(\chi_1^{n}\chi)\overline{\chi_1}^n,
\]
and then
\[
\langle(\overline{\chi_1}^n \otimes \psi_n)\chi,\overline{\xi}\rangle=\mu(\chi_1^{n-1}\chi)a(\chi_1^{n}\chi)\langle\overline{\chi_1}^n,\overline{\xi}\rangle
=
\begin{cases}
0, \xi\notin X_+^i\setminus\{\chi_0\};\\
\mu(\chi_1^{n-1}\chi)a(\chi_1^{n}\chi),  \xi=\chi_1^n.
\end{cases}
\]
Therefore
\[
\langle A\chi,\overline{\xi}\rangle=
\begin{cases}
0, \xi\notin X_+^i\setminus\{\chi_0\};\\
\mu(\chi)a(\xi\chi); \xi\in X_+^i\setminus\{\chi_0\}.
\end{cases}
\]
On the other hand, if $\xi\notin X_+^i\setminus\{\chi_0\}$ we get $\xi\in X_+\setminus X_+^i$ and thus $\chi\xi\in X_+\setminus X_+^i$ by Lemma \ref{lem1} which implies $a(\chi\xi)=0$. So,  we have $\langle A\chi,\overline{\xi}\rangle=\mu(\chi)a(\xi\chi)$ in the case  $\xi\notin X_+^i\setminus\{\chi_0\}$, as well.
\end{proof}

\subsection{Examples of $\mu$-Hankel operators}

Let $G=\mathbb{T}$ the circle group. Each character of $\mathbb{T}$ has the form $\chi_k(z)=z^k$, $k\in\mathbb{ Z}$. Thus, in this case $X_+=\{\chi_k: k\in\mathbb{ Z}_+\}$,  $X_-=\{\overline{\chi_j}: j\in\mathbb{N}\}$.

\textbf{Example 1.} Let $q\in \mathbb{C}$,  $0<|q|<1$, and $\sigma$ a finite (in general complex) regular Borel  measure on the open unit disk $\mathbb{D}\subset \mathbb{C}$ such that $\int_{\mathbb{D}}d|\sigma|(\zeta)/(1-|\zeta|)<\infty$.
 Consider the following integral operator, at least for trigonometric polynomials  $f\in \mathrm{span}(X_+)$ of analytic type:
\begin{equation}\label{A}
\mathbf{A}f(z)=\int_{\mathbb{D}}\frac{f(q\zeta)}{z-\zeta}d\sigma(\zeta) \quad (|z|=1)
\end{equation}
and the sequence of moments of the measure  $\sigma$
\[
\gamma_n:=\int_{\mathbb{D}}\zeta^n d\sigma(\zeta)\quad (n\in \mathbb{Z}_+).
\]
For every $k\in \mathbb{Z}_+$ we have
\begin{eqnarray}\label{Achi}
\mathbf{A}\chi_k(z)&=&q^k\int_{\mathbb{D}}\frac{\zeta^k}{z-\zeta}d\sigma(\zeta)=
q^kz^{-1}\int_{\mathbb{D}}\frac{\zeta^k}{1-z^{-1}\zeta}d\sigma(\zeta)\\\nonumber
&=&q^kz^{-1}\int_{\mathbb{D}}\sum_{n=0}^\infty \zeta^{k+n}z^{-n}d\sigma(\zeta)=q^k\sum_{n=0}^\infty\int_{\mathbb{D}}\zeta^{k+n}d\sigma(\zeta)z^{-n-1}\\\nonumber
&=&q^k\sum_{n=0}^\infty\gamma_{k+n}\chi_{-n-1}(z).
\end{eqnarray}
Above the term-by-term integration of the series is legal, since for all $k\in \mathbb{Z}_+$ one has
\begin{eqnarray*}
\sum_{n=0}^\infty\int_{\mathbb{D}}|\zeta^{k+n}|d|\sigma|(\zeta)&\le& \sum_{n=0}^\infty\int_{\mathbb{D}}|\zeta|^{n}d|\sigma|(\zeta)\\
&=&\int_{\mathbb{D}}\sum_{n=0}^\infty|\zeta|^{n}d|\sigma|(\zeta)=\int_{\mathbb{D}}\frac{d|\sigma|(\zeta)}{1-|\zeta|}<\infty.
\end{eqnarray*}
Moreover, this estimate shows also that  for all $k\in \mathbb{Z}_+$
\[
\sum_{n=0}^\infty|\gamma_{k+n}|^2\le \left(\sum_{n=0}^\infty|\gamma_{k+n}|\right)^2\le \left(\sum_{n=0}^\infty \int_{\mathbb{D}}|\zeta^{k+n}|d|\sigma|(\zeta)\right)^2<\infty.
\]
Thus, $(\gamma_{k+n})_{n\in Z_+}\in \ell^2(\mathbb{Z}_+)$ for all $k\in \mathbb{Z}_+$, and formula (\ref{Achi}) implies that $\mathbf{A}\chi_k\in H^2_-(\mathbb{T})$ and
\begin{eqnarray}\label{MatrixA}
\langle \mathbf{A}\chi_k,\overline{\chi_j}\rangle=q^k\sum_{n=0}^\infty\gamma_{k+n}\langle\chi_{-n-1},\chi_{-j}\rangle=q^k\gamma_{k+j-1}.
\end{eqnarray}
It follows that $\mathbf{A}=A_{\mu,a}$ where $\mu(\chi_k)=q^k$ for $k\in \mathbb{Z}_+$ and $a(\chi_k)=\gamma_{k-1}$ for $k\in \mathbb{N}$.
Now Theorem \ref{thm01} (i) shows that operator $\mathbf{A}$ is bounded  (the sequence $\gamma=(\gamma_n)$ belongs to $\ell^2(\mathbb{Z}_+)$) and
\begin{eqnarray}\label{normA}
\|\mathbf{A}\|\le \frac{\|\gamma\|_{\ell^2}}{\sqrt{1-|q|^2}}.
\end{eqnarray}
Moreover,  in this case operator $\mathbf{A}$ is nuclear by Theorem \ref{thm3}.

Formula (\ref{15}) shows that $\langle\varphi_n,\overline{\chi_j}\rangle=\gamma_{n+j-1}$. Therefore
\[
\varphi_n(z)=\sum_{n=0}^\infty \gamma_{n+j-1} \overline{z}^j=\int_{\mathbb{D}}\frac{\zeta^{n-1}}{1-\overline{z}\zeta}d\sigma(\zeta).
\]
By Theorem \ref{thm3}
\[
\mathbf{A}f=\sum_{n=0}^\infty q^n\langle f,z^n\rangle \varphi_n.
\]

 The operator $\mathbf{A}$ can be transferred to an operator acting on $H^2(\mathbb{T})$. Indeed, consider the "flip" operator $Jf(z)=\overline{z}f(\overline{z})$, $f\in L^2(\mathbb{T})$. This is a unitary operator and an involution  on $L^2(\mathbb{T})$, and $JH^2_-(\mathbb{T})=H^2(\mathbb{T})$ (and also $JH^2(\mathbb{T})=H^2_-(\mathbb{T})$). Then the operator
\[
(J\mathbf{A})f(z)=\int_{\mathbb{D}}\frac{f(q\zeta)}{1-z\zeta}d\sigma(\zeta)
\]
is nuclear in $H^2(\mathbb{T})$ and its norm satisfies the estimate (\ref{normA}). Since
\[
\langle J\mathbf{A}\chi_k,\chi_j\rangle=\langle \mathbf{A}\chi_k,J\chi_j\rangle=q^k\gamma_{k+j}
\]
by (\ref{MatrixA}), one has
\[
\mathrm{tr}(J\mathbf{A})=\sum_{k=0}^\infty q^k\gamma_{2k}=\sum_{k=0}^\infty q^k\int_{\mathbb{D}}\zeta^{2k}d\sigma(\zeta)=\int_{\mathbb{D}}\frac{d\sigma(\zeta)}{1-q\zeta^2}.
\]

To finish this example recall that every function $g=\sum_{j\ge 1}c_j\overline{\chi_j}$ from $H^2_-(\mathbb{T})$ extends to an anti-analytic function $g(z)=\sum_{j\ge 1}c_j\overline{z}^j$
on $\mathbb{D}$ with the  property $g(0)=0$. We shall show that the adjoint operator for $\mathbf{A}$ is as follows
\begin{eqnarray*}
\mathbf{A}^*g(z)&=&\int_{\mathbb{D}}\left(\overline{q}\frac{g(\zeta)}{\overline{z}-\overline{q\zeta}}+
\frac{g(q\zeta)}{\overline{q\zeta}}\right)d\overline{\sigma}(\zeta)\\
&=&\overline{q}\int_{\mathbb{D}}\frac{g(\zeta)}{\overline{z}-\overline{q\zeta}}d\overline{\sigma}(\zeta)
+\left(\int_{\mathbb{D}}\frac{g(q\zeta)}{\overline{q\zeta}}d\overline{\sigma}(\zeta)\right)\chi_0(z)\\
&=&Bg(z)+Cg(z) \qquad  (g\in H^2_-(\mathbb{T})).
\end{eqnarray*}
Indeed,
\[
\langle\chi_k,B\overline{\chi_j}\rangle=q\int_{\mathbb{T}}z^k\int_{\mathbb{D}}
\frac{\zeta^j}{z-q\zeta}d\sigma(\zeta)dm(z)=q\int_{\mathbb{D}}\zeta^j\int_{\mathbb{T}}\frac{z^k}{z-q\zeta}dm(z)d\sigma(\zeta)
\]
where $m$ stands for the normalized Haar measure on $\mathbb{T}$. Since
$$
\int_{\mathbb{T}}f(z)dm(z)=\frac{1}{2\pi}\int_0^{2\pi}f(e^{it})dt=\frac{1}{2\pi i}\int_{\mathbb{T}}f(z)\frac{dz}{z},
$$
one has for $k\in \mathbb{Z}_+$
\[
\int_{\mathbb{T}}\frac{z^k}{z-q\zeta}dm(z)=\frac{1}{2\pi i}\int_{\mathbb{T}}\frac{z^{k-1}}{z-q\zeta}dz=
\begin{cases}
(q\zeta)^{k-1}, k\ge 1;\\
0, k=0.
\end{cases}
\]
Thus,
\[
\langle\chi_k,B\overline{\chi_j}\rangle=
\begin{cases}
q^k\gamma_{k+j-1}, k\ge 1;\\
0, k=0.
\end{cases}
\]

On the other hand,
\[
\langle\chi_k,C\overline{\chi_j}\rangle=
\begin{cases}
0, k\ge 1;\\
q^{j-1}\gamma_{j-1}, k=0.
\end{cases}
\]
Now formula (\ref{MatrixA}) yields, that
 \[
\langle\mathbf{A}\chi_k,\overline{\chi_j}\rangle=\langle\chi_k,(B+C)\overline{\chi_j}\rangle,\ \ k\in\mathbb{ Z}_+, j\in \mathbb{N}.
\]
This formula and (\ref{MatrixA})  show  that the operator $B+C$ has the same matrix $(a_{jk})$ with respect to the standard bases of spaces $H^2_-(\mathbb{T})$ and $H^2_+(\mathbb{T})$ as $\mathbf{A}^*$ and  that $\sum_{j,k}|a_{jk}|^2<\infty$. It follows that $B+C$ is bounded (in fact, it is Hilbert--Schmidt, see, e.~g., \cite[Theorem 6.22]{Weidmann}) and thus $B+C=\mathbf{A}^*$.

\textbf{Example 2.} Consider an operator $\mathbf{A}$ given by the formula (\ref{A}) where $|q|=1$.
The operator $Uf(z)=f(qz)$ is unitary in $H^2(\mathbb{T})$ and $\mathbf{A}=HU$, where the operator
\begin{equation*}
Hf(z):=\int_{\mathbb{D}}\frac{f(\zeta)}{z-\zeta}d\sigma(\zeta)
\end{equation*}
between $H^2(\mathbb{T})$ and $H^2_-(\mathbb{T})$ is Hankel. Indeed, let $Jf(z)=\overline{z}f(\overline{z})$, $f\in L^2(\mathbb{T})$ be as in Example 1.
Then $H=J\Gamma$ where
\begin{equation*}
\Gamma f(z):=\int_{\mathbb{D}}\frac{f(\zeta)}{1-z\zeta}d\sigma(\zeta)
\end{equation*}
is a Hankel operator in  $H^2(\mathbb{T})$ \cite[Section 6.3]{Nik}. Thus, the operator $\mathbf{A}$ between $H^2(\mathbb{T})$ and $H^2_-(\mathbb{T})$ is bounded if and only if $\Gamma$ is bounded in $H^2(\mathbb{T})$. Let
the measure $\sigma$ be positive and supported in $(-1,1)$.  According to a theorem of Widom (see, e.~g., \cite[Theorem 6.2.1 and subsection 6.3.1]{Nik}) in this case the operator $\Gamma$ is bounded in $H^2(\mathbb{T})$ if and only if there is a constant $b>0$ such that $\gamma_n<b/(n+1)$, $n\in \mathbb{Z}_+$. For a complex measure  $\sigma$ on $\mathbb{D}$ the property  to be a Carleson measure is sufficient for the boundedness of $\Gamma$ in $H^2(\mathbb{T})$ (see, e.~g., \cite[pp. 314--315]{Nik}). As was mentioned in Theorem \ref{thm2} if the operator $\mathbf{A}$ is bounded it is $\mu$-Hankel where $\mu(\chi_k)=q^k$ ($k\in \mathbb{Z}_+$).

\section{$\nu$-Hankel operators on $G$}

In this section we briefly discuss the class of operators that are in some sense "dual" to  $\mu$-Hankel operators. Recall the definition of such operators.

 Let  $\nu$  be a non-null homomorphism of the semigroup $X_+$ into a multiplicative semigroup $\mathbb{C}$ and  $a: X_+\setminus\{\chi_0\}\to\mathbb{C}$. A  $\nu$-Hankel operator is a linear operator $B=B_{\nu,a}$ between the spaces  $H^2(G)$ and  $H^2_-(G)$
that is  defined at least on the linear subspace $\mathrm{span}(X_+)$ of $H^2(G)$ and satisfies
$$
\langle B_{\nu,a}\chi,\overline{\xi}\rangle=\nu(\xi)a(\chi\xi)\ \ \forall \chi\in X_+,  \xi\in  X_+\setminus\{\chi_0\}.
$$

\begin{theorem}\label{eq:nu} A bounded  $\nu$-Hankel operator $B: H^2(G)\to H^2_-(G)$  satisfies the operator equation
    \begin{equation}\label{op:B}
    S_{\overline{\chi}}^*B=\nu(\chi)BS_\chi \ \ \forall \chi\in X_+.
    \end{equation}
Conversely, if $\nu(\chi)\ne 0$ for all $\chi\in X_+$ then each bounded   operator $B: H^2(G)\to H^2_-(G)$ that satisfies (\ref{op:B}) is $\nu$-Hankel.
\end{theorem}

\begin{proof}
If a bounded   operator $B: H^2(G)\to H^2_-(G)$  is $\nu$-Hankel then for all $\chi, \eta\in X_+$, $\xi\in X_+\setminus\{\chi_0\}$ we have
\begin{eqnarray*}
\langle S_{\overline{\chi}}^*B\eta,\overline{\xi}\rangle&=&\langle B\eta,\overline{\chi\xi}\rangle=\nu(\chi\xi)a(\eta\chi\xi)\\
&=&
\nu(\chi)(\nu(\xi)a((\eta\chi)\xi))
=\nu(\chi)\langle BS_\chi\eta,\overline{\xi}\rangle
\end{eqnarray*}
and (\ref{op:B}) follows.

Now let $\nu(\chi)\ne 0$ for all $\chi\in X_+$ and a bounded   operator $B: H^2(G)\to H^2_-(G)$  satisfies (\ref{op:B}). If we put $a(\xi):=\frac{1}{\nu(\xi)}\langle B\chi_0, \overline{\xi}\rangle$ then for all $\chi\in X_+$,  $\xi\in  X_+\setminus\{\chi_0\}$ one has
\begin{eqnarray*}
\langle B\chi, \overline{\xi}\rangle&=&\langle BS_\chi \chi_0, \overline{\xi}\rangle=\frac{1}{\nu(\chi)}\langle  S_{\overline{\chi}}^*B\chi_0, \overline{\xi}\rangle\\
&=&\frac{1}{\nu(\chi)}\langle  B\chi_0, \overline{\chi\xi}\rangle=\nu(\xi)\frac{1}{\nu(\chi\xi)}\langle  B\chi_0, \overline{\chi\xi}\rangle=\nu(\xi)a(\chi\xi).
\end{eqnarray*}
\end{proof}

\textbf{Remark 2.} Theorem \ref{eq:nu} enables us to apply lemmas \ref{lemeq1}, \ref{lemeq2} to prove for $\nu$-Hankel operators analogs of the structure theorems \ref{thm3} and \ref{thm4}. 
If necessary, the reader can easily do this work. If $\nu(\xi)\ne 0$ $\forall \xi\in X_+$, we have $A_{(\mu;\nu),a}=A_{\mu',a\nu}$, where $\mu'=\frac{\mu}{\nu}$. Thus, analogs of theorems \ref{thm3} and \ref{thm4} for $A_{(\mu;\nu),a}$ are corollaries of this theorems.

\begin{theorem}\label{Bdd:nu} Assume that  a group $X$ possesses the first positive element $\chi_1$ and
 let $\nu\in \ell^2(X_+)$. A  $\nu$-Hankel operator $B_{\nu,a}$ extends from $\mathrm{span}(X_+)$ to a bounded  operator between the spaces  $H^2(G)$ and  $H^2_-(G)$
if and only if $a\in \ell^2(X_+)$. In this case the extension is unique and
$$
\|B_{\nu,a}\|\le \|\nu\|_{\ell^2(X_+\setminus\{\chi_0\})} \|a\|_{\ell^2(X_+\setminus\{\chi_0\})}.
$$
\end{theorem}

\begin{proof}
 Let $B_{\nu,a}$ be bounded. Then for all $\chi\in X_+$ the numbers
 $$
 \langle  B_{\nu,a}^*\overline{\chi_1}, \chi\rangle=\overline{\langle \chi, B_{\nu,a}^*\overline{\chi_1}\rangle}=\overline{\langle  B_{\nu,a}\chi, \overline{\chi_1}\rangle}=\overline{\nu(\chi_1)a(\chi\chi_1)}
 $$
 are Fourier coefficients of the function $ B_{\nu,a}^*\overline{\chi_1}$. Thus, the function $\chi\mapsto a(\chi\chi_1)$ ($\chi\in X_+$) belongs to $\ell^2(X_+)$.  On the other hand, $\{\chi\chi_1:\chi\in X_+\}= X_+\setminus \{\chi_0\}$. (Indeed, $\xi\in  X_+\setminus \{\chi_0\}$ means that $\xi\ge\chi_1$. In turn, this is equivalent to $\chi:=\xi\chi_1^{-1}\in X_+$ and so $\xi=\chi\chi_1$.) It follows, that   $a\in \ell^2(X_+\setminus\{\chi_0\})$.

Now let $\nu\in \ell^2(X_+)$, $a\in \ell^2(X_+)$ and $B=B_{\nu,a}$. As in the proof of the Theorem \ref{thm01} for each polynomial $f\in \mathrm{span}(X_+)$ we have $Bf=\sum_{\chi\in X_+}\langle f,\chi\rangle B\chi$ (a finite sum). It follows that
\begin{eqnarray*}
Bf&=&\sum_{\xi\in X_+\setminus\{\chi_0\}}\langle Bf,\overline{\xi}\rangle \overline{\xi}=\sum_{\xi\in X_+\setminus\{\chi_0\}}\left(\sum_{\chi\in X_+}\langle f,\chi\rangle \langle B\chi,\overline{\xi}\rangle\right)\overline{\xi}\\
&=&\sum_{\xi\in X_+\setminus\{\chi_0\}}\left(\sum_{\chi\in X_+}\langle f,\chi\rangle \nu(\xi)a(\chi\xi)\right)\overline{\xi}.
\end{eqnarray*}
Therefore by Parseval's formula and Cauchy-Bunyakovskii-Schwartz inequality one has
\begin{eqnarray*}
\|Bf\|^2&=&\sum_{\xi\in X_+\setminus\{\chi_0\}}\left|\sum_{\chi\in X_+}\langle f,\chi\rangle \nu(\xi)a(\chi\xi)\right|^2\\
&\le&\sum_{\xi\in X_+\setminus\{\chi_0\}}|\nu(\xi)|^2\left(\sum_{\chi\in X_+}|\langle f,\chi\rangle|^2\right)\left(\sum_{\chi\in X_+}|a(\chi\xi)|^2\right)\\
&\le&\|\nu\|_{\ell^2(X_+\setminus\{\chi_0\})}^2\|a\|_{\ell^2(X_+\setminus\{\chi_0\})}^2\|f\|^2.
\end{eqnarray*}
As in the proof of the Theorem \ref{thm01}
 it remains to note that $\mathrm{span}(X_+)$ is dense in $H^2(G)$ by \cite[Lemma 1]{MathSb}.

\end{proof}

\textbf{Example 3.} Let $q\in \mathbb{C}$,  $|q|>1$, and $\tau$ a finite (in general complex) regular Borel  measure on the closed unit disk $\overline{\mathbb{D}}\subset \mathbb{C}$ with sequence of moments $\gamma=(\gamma_n)_{n\in \mathbb{Z}_+}$. Consider the operator
\[
\mathbf{B}f(z):=\int_{\overline{\mathbb{D}}}\frac{f(\zeta)}{qz-\zeta }
d\tau(\zeta)\quad (|z|=1).
\]

For all $k\in\mathbb{ Z}_+$ we have
\begin{eqnarray*}
\mathbf{B}\chi_k(z)&=&\overline{z}\int_{\overline{\mathbb{D}}}\frac{\zeta^k}{q-\zeta \overline{z}}d\tau(\zeta)\\
&=&\frac{\overline{z}}{q}\int_{\overline{\mathbb{D}}}\zeta^k\sum_{n=0}^\infty\left(\frac{\zeta \overline{z}}{q}\right)^nd\tau(\zeta)
=\sum_{n=0}^\infty \frac{\gamma_{n+k}}{q^{n+1}}\overline{\chi_{n+1}(z)}.
\end{eqnarray*}
(The term-by-term integration of the series is legal, since $1/|q|<1$ and $(\gamma_{n})$ is  bounded.)
Since for all $k$ the sequence $(\gamma_{n+k}/q^{n+1})$ belongs to $\ell^2$, we have $\mathbf{B}\chi_k\in H^2_-(\mathbb{T})$. Moreover,
\[
\langle \mathbf{B}\chi_k,\overline{\chi_j}\rangle=\sum_{n=0}^\infty \frac{\gamma_{n+k}}{q^{n+1}}\langle\overline{\chi_{n+1}},\overline{\chi_j}\rangle=\frac{1}{q^{j}}\gamma_{j+k-1}.
\]
Thus, $\mathbf{B}$ is $\nu$-Hankel with $\nu(\chi_j)=\left(\frac{1}{q}\right)^j$, $a(\chi_k)=\gamma_{k-1}$. Now Theorem \ref{Bdd:nu}  implies that the operator  $\mathbf{B}$ is bounded if and only if $\gamma\in\ell^2$ and in this case 
\begin{equation}\label{normB}
\|\mathbf{B}\|\le\frac{\|\gamma\|_{\ell^2}}{\sqrt{|q|^2-1}}.
\end{equation}

Similar to the Example 1 the operator
\[
(J\mathbf{B})f(z)=\int_{\overline{\mathbb{D}}}\frac{f(\zeta)}{q-\zeta z}
d\tau(\zeta)
\]
is bounded in $H^2(\mathbb{T})$  if and only if $\gamma\in\ell^2$ and in this case the estimate (\ref{normB}) holds for its norm. Moreover, this operator is nuclear and
\[
\mathrm{tr}(J\mathbf{B})=\int_{\overline{\mathbb{D}}}\frac{d\tau(\zeta)}{q-\zeta^2}.
\] 

\section{Acknowledgments}   This work was  supported in part by the State Program of Scientific Research of the Republic of Belarus, project No. 20211776.


\end{document}